\input ./style/arxiv-vmsta.cfg
\documentclass[numbers,compress,v1.0.1]{vmsta}
\usepackage{vtexbibtags}

\volume{7}% Updated by VTEXPTS2LaTeX.exe, 25.03.2020 14:32
\issue{1}% Updated by VTEXPTS2LaTeX.exe, 25.03.2020 14:32
\pubyear{2020}
\firstpage{1}
\lastpage{15}
\aid{VMSTA150}% Updated by VTEXPTS2LaTeX.exe, 25.02.2020 12:14
\doi{10.15559/20-VMSTA150}% Updated by VTEXPTS2LaTeX.exe, 25.02.2020 12:14
\articletype{research-article}

\startlocaldefs
\newtheorem{thm}{Theorem}

\theoremstyle{definition}
\newtheorem{defin}{Definition}

\newtheorem{assumption}{Assumption}
\allowdisplaybreaks

\ifdefined\HCode % tex4ht
\def\index#1{}
\else % LaTeX
\fi

\endlocaldefs

\begin{document}

\begin{frontmatter}
\pretitle{Research Article}

\title{Stochastic two-species mutualism model with jumps}

\author[a]{\inits{Olg.}\fnms{Olga}~\snm{Borysenko}\ead[label=e1]{olga\_borisenko@ukr.net}}
\author[b]{\inits{O.}\fnms{Oleksandr}~\snm{Borysenko}\thanksref{cor1}\ead[label=e2]{odb@univ.kiev.ua}}
\thankstext[type=corresp,id=cor1]{Corresponding author.}
\address[a]{Department of Mathematical Physics,
\institution{National Technical University of Ukraine}, 37, Prosp. Peremohy, Kyiv, 03056, \cny{Ukraine}}

\address[b]{Department of Probability Theory, Statistics and Actuarial
Mathematics, \institution{Taras Shevchenko National University of Kyiv, Ukraine},
64 Volodymyrska Str., Kyiv, 01601 \cny{Ukraine}}

%\author[]{\inits{}\fnms{}~\snm{}\thanksref{f1}\ead[label=e1]{}}
%\author[]{\inits{}\fnms{}~\snm{}\thanksref{f1}\thanksref{cor1}\ead[label=e2]{}}
%\thankstext[type=corresp,id=cor1]{Corresponding author.}
%\address[]{\institution{}, ..., \cny{}}
%\address[]{\institution{}, ..., \cny{}}

%\thankstext[id=f1]{}

%\dedicated{}

\markboth{Olg.~Borysenko, O.~Borysenko}{Stochastic two-species mutualism model with jumps}
%\markboth{}{}

\begin{abstract}
The existence and uniqueness are proved for the global positive solution
to the system of stochastic differential equations describing a two-species
mutualism model disturbed by the white noise, the centered and non-centered Poisson
noises. We obtain sufficient conditions for stochastic ultimate boundedness,
stochastic permanence, nonpersistence in the mean, strong persistence
in the mean and extinction of the solution to the considered system.
\end{abstract}


\received{\sday{3} \smonth{12} \syear{2019}}% Updated by VTEXPTS2LaTeX.exe, 25.02.2020 12:14
\revised{\sday{21} \smonth{2} \syear{2020}}% Updated by VTEXPTS2LaTeX.exe, 25.02.2020 12:14
\accepted{\sday{21} \smonth{2} \syear{2020}}% Updated by VTEXPTS2LaTeX.exe, 25.02.2020 12:14
\publishedonline{\sday{3} \smonth{3} \syear{2020}}

\end{frontmatter}

%s1 #&#
\section{Introduction}
\label{sec1}

The construction of the mutualism model and its properties are presented
in K. Gopalsamy \cite{Gopal}. Mutualism occurs when one species provides
some benefit in exchange for another benefit. A deterministic two-species
mutualism model is described by the system
\begin{align}
\frac{dN_{1}(t)}{dt}=r_{1}(t)N_{1}(t)\left [
\frac{K_{1}(t)+\alpha _{1}(t)N_{2}(t)}{1+N_{2}(t)}-N_{1}(t)\right ],
\nonumber
\\
\frac{dN_{2}(t)}{dt}=r_{2}(t)N_{2}(t)\left [
\frac{K_{2}(t)+\alpha _{2}(t)N_{1}(t)}{1+N_{1}(t)}-N_{2}(t)\right ],
\nonumber
\end{align}
where $N_{1}(t)$ and $N_{2}(t)$ denote the population densities of each species
at time $t$, $r_{i}(t)>0$, $i=1,2$, denotes the intrinsic growth rate of species
$N_{i},i=1,2$, and $\alpha _{i}(t)>K_{i}(t)>0$, $i=1,2$. The carrying capacity
of species $N_{i}(t)$ is $K_{i}(t)$, $i=1,2$, in the absence of other species.
In the paper by Hong Qiu, Jingliang Lv and Ke Wang \cite{Qiu}
the stochastic mutualism model\index{stochastic ! mutualism model} of the form
%
%e1 #&#
\begin{align}
\label{eq1}
dx(t)=x(t)\left [\frac{a_{1}(t)+a_{2}(t)y(t)}{1+y(t)}-c_{1}(t)x(t)
\right ]+\sigma _{1}(t)x(t)dw_{1}(t),
\nonumber
\\
dy(t)=y(t)\left [\frac{b_{1}(t)+b_{2}(t)x(t)}{1+x(t)}-c_{2}(t)y(t)
\right ]+\sigma _{2}(t)y(t)dw_{2}(t)
\end{align}
is considered, where $a_{i}(t), b_{i}(t), c_{i}(t), \sigma _{i}(t), i =1,2$, are all positive,
continuous and bounded functions on $[0,+\infty )$, and
$w_{1}(t), w_{2}(t)$ are independent Wiener processes. The authors show
that the stochastic system \textup{(\ref{eq1})} has a unique global (no explosion
in a finite time) solution for any positive initial value and that this
stochastic model is stochastically ultimately bounded. The sufficient conditions
for stochastic permanence\index{stochastic ! permanence} and persistence in the mean of the solution to
the system \textup{(\ref{eq1})} are obtained.

Population systems may suffer abrupt environmental perturbations,\index{abrupt environmental perturbations} such
as epidemics, fires, earthquakes, etc. So it is natural to introduce Poisson
noises into the population model for describing such discontinuous systems.
In the paper by Mei Li, Hongjun Gao and Binjun Wang \cite{Li} the authors
consider the stochastic mutualism model\index{stochastic ! mutualism model} with the white and centered Poisson
noises:
\begin{align}
dx(t)&=x(t^{-})\left [\rule{0pt}{20pt}\left (r_{1}(t)-
\frac{b_{1}(t)x(t)}{K_{1}(t)+y(t)}-\varepsilon _{1}(t)x(t)\right )dt
\right .
\nonumber
\\
&+\alpha _{1}(t)dw_{1}(t)+\left .\int \limits _{\mathbb{Y}}\gamma _{1}(t,z)
\tilde{\nu }(dt,dz)\right ],
\nonumber
\\
dy(t)&=y(t^{-})\left [\rule{0pt}{20pt}\left (r_{2}(t)-
\frac{b_{2}(t)y(t)}{K_{2}(t)+x(t)}-\varepsilon _{2}(t)y(t)\right )dt
\right .
\nonumber
\\
&+\alpha _{2}(t)dw_{2}(t)+\left .\int \limits _{\mathbb{Y}}\gamma _{2}(t,z)
\tilde{\nu }(dt,dz)\right ],
\nonumber
\end{align}
where $x(t^{-})$, $y(t^{-})$ are the left limit of $x(t)$ and $y(t)$ respectively,
$r_{i}(t)$, $b_{i}(t)$, $K_{i}(t)$, $\alpha _{i}(t)$, $i=1,2$, are all positive,
continuous and bounded functions, $\mathbb{Y}$ is measurable subset of
$(0,+\infty )$, $w_{i}(t)$, $i=1,2$, are independent standard\index{independent standard} one-dimensional
Wiener processes, $\tilde{\nu }(t,A)=\nu (t,A)-t\Pi (A)$ is the centered
Poisson measure independent on $w_{i}(t)$, $i=1,2$,
$E[\nu (t,A)]=t\Pi (A)$, $\Pi (\mathbb{Y})<\infty $,
$\gamma _{i}(t,z),i=1,2$, are random, measurable, bounded, continuous in
$t$. The global existence and uniqueness of the positive solution to this
problem are proved. The sufficient conditions of stochastic boundedness,
stochastic permanence,\index{stochastic ! permanence} persistence\index{persistence} in the mean and extinction\index{extinction} of the solution
are obtained.

In this paper, we consider the stochastic mutualism model\index{stochastic ! mutualism model} with jumps generated
by centered and noncentered Poisson measures. So, we take into account
not only ``small'' jumps, corresponding to the centered Poisson measure,
but also ``large'' jumps, corresponding to the noncentered Poisson
measure. This model is driven by the system of stochastic differential
equations\index{stochastic ! differential equations}
%
%e2 #&#
\begin{align}
\label{eq2}
dx_{i}(t)=x_{i}(t)\left [
\frac{a_{i1}(t)+a_{i2}(t)x_{3-i}(t)}{1+x_{3-i}(t)}-c_{i}(t)x_{i}(t)
\right ]dt+\sigma _{i}(t)x_{i}(t)dw_{i}(t)
\nonumber
\\
+\int \limits _{\mathbb{R}}\gamma _{i}(t,z)x_{i}(t)\tilde{\nu }_{1}(dt,dz)+
\int \limits _{\mathbb{R}}\delta _{i}(t,z)x_{i}(t)\nu _{2}(dt,dz),
\nonumber
\\
x_{i}(0)=x_{i0}>0,\quad  i=1,2\xch{,}{.}
\end{align}
where $w_{i}(t), i=1,2$, are independent standard\index{independent standard} one-dimensional Wiener
processes, $\tilde{\nu }_{1}(t,A)=\nu _{1}(t,A)-t\Pi _{1}(A)$,
$\nu _{i}(t,A), i=1,2$, are independent Poisson measures,\index{independent Poisson measures} which are independent
on $w_{i}(t),i=1,2$, $E[\nu _{i}(t,A)]=t\Pi _{i}(A)$, $i=1,2$,
$\Pi _{i}(A), i=1,2$, are  finite measures\index{finite measures} on the Borel sets $A$ in
$\mathbb{R}$.

To the best of our knowledge, there are no papers devoted to the dynamical
properties of the stochastic mutualism model\index{stochastic ! mutualism model} \textup{(\ref{eq2})}, even in the
case of the centered Poisson noise. It is worth noting that the impact of the centered
and noncentered Poisson noises to the stochastic nonautonomous logistic
model is studied in the papers by O.D. Borysenko and D.O. Borysenko
\cite{Bor1,Bor2}.

In the following we will use the notations
$X(t)=(x_{1}(t),x_{2}(t))$, $X_{0}=(x_{10},x_{20})$,
$|X(t)|=\sqrt{x_{1}^{2}(t)+x_{2}^{2}(t)}$,
$\mathbb{R}^{2}_{+}=\{X\in \mathbb{R}^{2}:\ x_{1}>0,x_{2}>0\}$,
%
%e3 #&#
\begin{equation}
\beta _{i}(t)=\sigma ^{2}_{i}(t)/2+\int \limits _{\mathbb{R}}[
\gamma _{i}(t,z)-\ln (1+\gamma _{i}(t,z))]\Pi _{1}(dz)-\int \limits _{
\mathbb{R}}\ln (1+\delta _{i}(t,z))]\Pi _{2}(dz),
\nonumber
\end{equation}
$i=1,2$. For the bounded, continuous functions
$f_{i}(t), t\in [0,+\infty )$, $i=1,2$, let us denote
\begin{align}
f_{i\sup }=\sup _{t\ge 0}f_{i}(t),\qquad  f_{i\inf }=\inf _{t\ge 0}f_{i}(t),\quad  i=1,2,
\nonumber
\\
f_{\max }=\max \{f_{1\sup },f_{2\sup }\},\qquad  f_{\min }=\min \{f_{1\inf },f_{2
\inf }\}.
\nonumber
\end{align}

We will prove that system \textup{(\ref{eq2})} has a unique, positive, global (no explosion
in a finite time) solution for any positive initial value, and that this
solution is stochastically ultimate bounded. The sufficient conditions
for stochastic permanence,\index{stochastic ! permanence} nonpersistence in the mean, strong persistence
in the mean and extinction\index{extinction} of solution are derived.

The rest of this paper is organized as follows. In Section~\ref{sec2}, we prove
the existence of the unique global positive solution to the system
\textup{(\ref{eq2})}. In Section~\ref{sec3}, we prove the stochastic ultimate boundedness
of the solution to the system \textup{(\ref{eq2})}. In Section~\ref{sec4}, we obtain
conditions under which the solution to the system \textup{(\ref{eq2})} is stochastically
permanent,\index{stochastically permanent} and in Section~\ref{sec5} the sufficient conditions for nonpersistence
in the mean, strong persistence in the mean and extinction\index{extinction} of the solution
are obtained.

%s2 #&#
\section{Existence of the global solution}
\label{sec2}

Let $(\Omega ,{\mathcal{F}},\mathrm{P})$ be a probability space,
$w_{i}(t), i=1$, $2, t\ge 0$, are independent standard\index{independent standard} one-dimensional Wiener
processes on $(\Omega ,{\mathcal{F}},\mathrm{P})$, and
$\nu _{i}(t,A), i=1,2$, are independent Poisson measures\index{independent Poisson measures} defined on
$(\Omega ,{\mathcal{F}},\mathrm{P})$ independent on $w_{i}(t), i=1,2$. Here
$\mathrm{E}[\nu _{i}(t,A)]=t\Pi _{i}(A)$, $i=1,2$,
$\tilde{\nu }_{i}(t,A)=\nu _{i}(t,A)-t\Pi _{i}(A)$, $i=1,2$,
$\Pi _{i}(\cdot ), i=1,2$, are finite measures\index{finite measures} on the Borel sets in
$\mathbb{R}$. On the probability space
$(\Omega ,{\mathcal{F}},\mathrm{P})$ we consider an increasing, right-continuous
family of complete sub-$\sigma $-algebras
$\{{\mathcal{F}}_{t}\}_{t\ge 0}$, where
${\mathcal{F}}_{t}=\sigma \{w_{i}(s),\nu _{i}(s,A), s\le t,i=1,2\}$.

We need the following assumption.

%a1 #&#
\begin{assumption}
\label{ass1}
It is assumed that $a_{ij}(t),i,j=1,2$,
$c_{i}(t), \sigma _{i}(t), i=1,2$, are bounded, continuous in $t$ functions,
$a_{ij}(t)>0$, $i,j=1,2$, $c_{\min }>0$,
$\gamma _{i}(t,z), \delta _{i}(t,z)$, $i=1,2$, are continuous in $t$ functions
and $\ln (1+\gamma _{i}(t,z)),\ln (1+\delta _{i}(t,z))$, $i=1,2$, are bounded,
$\Pi _{i}(\mathbb{R})<\infty $, $i=1,2$.%
\end{assumption}

%t1 #&#
\begin{thm}
\label{thm1}
Let Assumption \ref{ass1} be fulfilled. Then there exists a unique global
solution $X(t)$ to the system \textup{(\ref{eq2})} for any initial value
$X(0)=X_{0}>0$, and $\mathrm{P}\{X(t)\in \mathbb{R}^{2}_{+}\}=\nobreak 1$.
\end{thm}

\begin{proof}
Let us consider the system of stochastic differential equations\index{stochastic ! differential equations}
%
%e4 #&#
\begin{align}
\label{eq3}
dv_{i}(t)=\left [
\frac{a_{i1}(t)+a_{i2}(t)\exp \{v_{3-i}(t)\}}{1+\exp \{v_{3-i}(t)\}}-c_{i}(t)
\exp \{v_{i}(t)\}-\beta _{i}(t)\right ]dt
\nonumber
\\
+\sigma _{i}(t)dw_{i}(t) +\int \limits _{\mathbb{R}}\ln (1+\gamma _{i}(t,z))
\tilde{\nu }_{1}(dt,dz)+\int \limits _{\mathbb{R}}\ln (1+\delta _{i}(t,z))
\tilde{\nu }_{2}(dt,dz),
\nonumber
\\
v_{i}(0)=\ln x_{i0},\quad  i=1,2.
\end{align}

The coefficients of  equation \textup{(\ref{eq3})} are locally Lipschitz continuous.\index{locally Lipschitz continuous}
Therefore, for any initial value $(v_{1}(0),v_{2}(0))$ there exists a unique
local solution $V(t)=(v_{1}(t),\allowbreak v_{2}(t))$ on $[0,\tau _{e})$, where
$\sup _{t<\tau _{e}}|V(t)|=+\infty $ (cf. Theorem 6, p.~246,~\cite{GikhSkor}).
So, from the It\^{o} formula\index{It\^{o} formula} we derive that the process
$X(t)=(\exp \{v_{1}(t)\},\exp \{v_{2}(t)\})$ is a unique, positive local
solution to the system (\ref{eq2}). To show this solution is global,
we need to show that $\tau _{e}=+\infty $ a.s. Let
$n_{0}\in \mathbb{N}$ be sufficiently large for
$x_{i0}\in [1/n_{0},n_{0}]$, $i=1,2$. For any $n\ge n_{0}$ we define the
stopping time
\begin{align}
\tau _{n}=\inf \left \{  t\in [0,\tau _{e}):\ X(t)\notin \left (
\frac{1}{n},n\right )\times \left (\frac{1}{n},n\right )\right \}  .
\nonumber
\end{align}
It is clear that $\tau _{n}$ is increasing as $n\to +\infty $. Set
$\tau _{\infty }=\lim _{n\to \infty }\tau _{n}$, whence
$\tau _{\infty }\le \tau _{e}$ a.s. If we prove that
$\tau _{\infty }=\infty $ a.s., then $\tau _{e}=\infty $ a.s. and
$X(t)\in \mathbb{R}^{2}_{+}$ a.s. for all $t\in [0,+\infty )$. So we need
to show that $\tau _{\infty }=\infty $ a.s. If this statement is false, there
are constants $T>0$ and $\varepsilon \in (0,1)$, such that
$\mathrm{P}\{\tau _{\infty }<T\}>\varepsilon $. Hence, there is
$n_{1}\ge n_{0}$ such that
%
%e5 #&#
\begin{align}
\label{eq4}
\mathrm{P}\{\tau _{n}<T\}>\varepsilon ,\quad \forall n\ge n_{1}.
\end{align}
For the nonnegative function
$V(X)=\sum \limits _{i=1}^{2}(x_{i}-1-\ln x_{i})$, $x_{i}>0$,
$i=1,2$, by the It\^{o} formula\index{It\^{o} formula} we have
%
%e6 #&#
\begin{align}
\label{eq5}
V(X(T\wedge \tau _{n}))=V(X_{0})+\sum _{i=1}^{2}\left \{  \int \limits _{0}^{T
\wedge \tau _{n}}\left [\rule{0pt}{20pt}(x_{i}(t)-1)\left (
\frac{a_{i1}(t)+a_{i2}(t)x_{3-i}(t)}{1+x_{3-i}(t)}\right .\right .
\right .
\nonumber
\\
\left .\left .-c_{i}(t)x_{i}(t)\rule{0pt}{14pt}\right )+ \beta _{i}(t)+
\int \limits _{\mathbb{R}}\delta _{i}(t,z)x_{i}(t)\Pi _{2}(dz)
\right ]dt
\nonumber
\\
+\int \limits _{0}^{T\wedge \tau _{n}}\!\!\!(x_{i}(t)-1)
\sigma _{i}(t)dw_{i}(t)\! +\!\int \limits _{0}^{T\wedge \tau _{n}}
\!\!\!\int \limits _{\mathbb{R}}\![\gamma _{i}(t,z)x_{i}(t)-\ln (1+
\gamma _{i}(t,z))]\tilde{\nu }_{1}(dt,dz)
\nonumber
\\
\left .+ \int \limits _{0}^{T\wedge \tau _{n}}\!\!\!\int \limits _{
\mathbb{R}}[\delta _{i}(t,z)x_{i}(t)-\ln (1+\delta _{i}(t,z))]
\tilde{\nu }_{2}(dt,dz)\right \}  .
\end{align}

Under the conditions of the theorem there is a constant $K>0$ such that
%
%e7 #&#
\begin{align}
\label{eq6}
\sum _{i=1}^{2}\left [(x_{i}-1)\left (
\frac{a_{i1}(t)+a_{i2}(t)x_{3-i}}{1+x_{3-i}}-c_{i}(t)x_{i}\right )+
\beta _{i}(t)\right .
\nonumber
\\
\left .+\int \limits _{\mathbb{R}}\delta _{i}(t,z)x_{i}\Pi _{2}(dz)
\right ]\le \sum _{i=1}^{2}\left [(a_{\max }+c_{\max })x_{i}-c_{\min }x_{i}^{2}+
\beta _{\max }
\nonumber
\right .
\\
\left .+ x_{i}\delta _{\max }\Pi _{2}(\mathbb{R})\right ]\le K,
\end{align}
where $a_{\max }=\max _{i,j=1,2}\{\sup _{t\ge 0}a_{ij}(t)\}$. From
\textup{(\ref{eq5})} and \textup{(\ref{eq6})} we have
\begin{align}
V(X(T\wedge \tau _{n}))\le V(X_{0})+K(T\wedge \tau _{n})+\sum _{i=1}^{2}
\left \{  \int \limits _{0}^{T\wedge \tau _{n}}(x_{i}(t)-1)\sigma _{i}(t)dw_{i}(t)
\right .
\nonumber
\\
+\left . \int \limits _{0}^{T\wedge \tau _{n}}\!\!\!\int \limits _{
\mathbb{R}}[\gamma _{i}(t,z)x_{i}(t)-\ln (1+\gamma _{i}(t,z))]
\tilde{\nu }_{1}(dt,dz)\right .
\nonumber
\\
+\left . \int \limits _{0}^{T\wedge \tau _{n}}\!\!\!\int \limits _{
\mathbb{R}}[\delta _{i}(t,z)x_{i}(t)-\ln (1+\delta _{i}(t,z))]
\tilde{\nu }_{2}(dt,dz)\right \}  .
\nonumber
\end{align}

Whence, taking expectations yields
%
%e8 #&#
\begin{align}
\label{eq7}
\mathrm{E}\left [V(X(T\wedge \tau _{n}))\right ]\le V(X_{0})+KT.
\end{align}

Set $\Omega _{n}=\{\tau _{n}\le T\}$ for $n\ge n_{1}$. Then by
\textup{(\ref{eq4})},
$\mathrm{P}(\Omega _{n})=\mathrm{P}\{\tau _{n}\le t\}>\varepsilon $,
$\forall n\ge n_{1}$. Note that for every $\omega \in \Omega _{n}$ there
is some $i$ such that $x_{i}(\tau _{n},\omega )$ equals either $n$ or
$1/n$. So
\begin{align}
V(X(\tau _{n}))\ge \min \{n-1-\ln n,\frac{1}{n}-1+\ln n\}.
\nonumber
\end{align}

It then follows from \textup{(\ref{eq7})} that
\begin{align}
V(X_{0})+KT\ge \mathrm{E}[\mathbf{1}_{\Omega _{n}}V(X(\tau _{n}))]
\ge \varepsilon \min \{n-1-\ln n,\frac{1}{n}-1+\ln n\},
\nonumber
\end{align}
where $\mathbf{1}_{\Omega _{n}}$ is the indicator function of
$\Omega _{n}$. Letting $n\to \infty $ leads to the contradiction
$\infty >V(X_{0})+KT=\infty $. This completes the proof of the theorem.
\end{proof}

%s3 #&#
\section{Stochastically ultimate boundedness}
\label{sec3}

%d1 #&#
\begin{defin}%
\label{def1}%
(\cite{LiMao}) The solution $X(t)$ to the system \textup{(\ref{eq2})} is said
to be stochastically ultimately bounded, if for any
$\varepsilon \in (0,1)$ there is a positive constant
$\chi =\chi (\varepsilon )>0$ such that for any initial value
$X_{0}\in \mathbb{R}^{2}_{+}$
%the solution to the system \textup{(\ref{eq2})}
this solution has the property
\begin{align}
\limsup _{t\to \infty }\mathrm{P}\{|X(t)|>\chi \}<\varepsilon .
\nonumber
\end{align}
\end{defin}

%t2 #&#
\begin{thm}%
\label{thm2}
Under Assumption \ref{ass1} the solution $X(t)$ to the system
\textup{(\ref{eq2})} is stochastically ultimately bounded for any initial value
$X_{0}\in \mathbb{R}^{2}_{+}$.
\end{thm}

\begin{proof}
Let $\tau _{n}$ be the stopping time defined in Theorem
\ref{thm1}. Applying the It\^{o} formula\index{It\^{o} formula} to the process
$V(t,x_{i}(t))=e^{t}x_{i}^{p}(t)$, $i=1,2$, $p>0$, we obtain for
$i=1,2$ that
%
%e9 #&#
\begin{align}
\label{eq8}
V(t\wedge \tau _{n},x_{i}(t\wedge \tau _{n}))=x_{i0}^{p}+\int
\limits _{0}^{t\wedge \tau _{n}}e^{s}x_{i}^{p}(s)\left \{
\rule{0pt}{18pt}1+p\left [
\frac{a_{i1}(s)+a_{i2}(s)x_{3-i}(s)}{1+x_{3-i}(s)}\right .\right .
\nonumber
\\
\left .\left .- c_{i}(s)x_{i}(s)\right ]\!+\!
\frac{p(p-1)\sigma _{i}^{2}(s)}{2}+\int \limits _{\mathbb{R}}\!\!\!
\left [(1+\gamma _{i}(s,z))^{p}\!-\!1\!-\!p\gamma _{i}(s,z)\right ]
\Pi _{1}(dz)\!\right .
\nonumber
\\
\left . +\!\int \limits _{\mathbb{R}}\!\!\!\left [(1+\delta _{i}(s,z))^{p}
\!-\!1\right ]\Pi _{2}(dz) \right \}  ds +\int \limits _{0}^{t\wedge
\tau _{n}}\!\!\!p e^{s}x_{i}^{p}(s)\sigma _{i}(s)dw_{i}(s)
\nonumber
\\
+\int \limits _{0}^{t\wedge \tau _{n}}\!\!\!\int \limits _{
\mathbb{R}}e^{s}x_{i}^{p}(s)\left [(1+\gamma _{i}(s,z))^{p}-1\right ]
\tilde{\nu }_{1}(ds,dz)
\nonumber
\\
+\int \limits _{0}^{t\wedge \tau _{n}}\!\!\!\int \limits _{
\mathbb{R}}e^{s}x_{i}^{p}(s)\left [(1+\delta _{i}(s,z))^{p}-1\right ]
\tilde{\nu }_{2}(ds,dz).
\end{align}

Under Assumption \ref{ass1} there is a constant $K_{i}(p)>0$ such that
%
%e10 #&#
\begin{align}
\label{eq9}
e^{s}x_{i}^{p}\left \{  \rule{0pt}{18pt}1+p\left [
\frac{a_{i1}(s)+a_{i2}(s)x_{3-i}}{1+x_{3-i}}- c_{i}(s)x_{i}\right ]\!+
\!\frac{p(p-1)\sigma _{i}^{2}(s)}{2}+\right .
\nonumber
\\
\left .+\!\int \limits _{\mathbb{R}}\!\!\left [(1+\gamma _{i}(s,z))^{p}
\!-\!1\!-\!p\gamma _{i}(s,z)\right ]\Pi _{1}(dz) +\!\int \limits _{
\mathbb{R}}\!\!\left [(1+\delta _{i}(s,z))^{p}\!-\!1\right ]\Pi _{2}(dz)
\right \}
\nonumber
\\
\le K_{i}(p)e^{s}.
\end{align}
From \textup{(\ref{eq8})} and \textup{(\ref{eq9})}, taking expectations, we obtain
\begin{align}
\mathrm{E}[V(t\wedge \tau _{n},x_{i}(t\wedge \tau _{n}))]\le x_{i0}^{p}+K(p)
\mathrm{E}[e^{t\wedge \tau _{n}}]\le x_{i0}^{p}+K_{i}(p)e^{t}.
\nonumber
\end{align}

Letting $n\to \infty $ leads to the estimate
%
%e11 #&#
\begin{align}
\label{eq10}
e^{t}\mathrm{E}[x_{i}^{p}(t)]\le x_{i0}^{p}+e^{t}K_{i}(p).
\end{align}
So we have for $i=1,2$ that
%
%e12 #&#
\begin{align}
\label{eq11}
\lim \sup _{t\to \infty }E[x_{i}^{p}(t)]\le K_{i}(p).
\end{align}
For $X=(x_{1},x_{2})\in \mathbb{R}^{2}_{+}$ we have
$|X|^{p}\le 2^{p/2}(x_{1}^{p}+x_{2}^{p})$, therefore, from
\textup{(\ref{eq11})}
$\lim \sup _{t\to \infty }E[|X(t)|^{p}]\le L(p)=2^{p/2}(K_{1}(p)+K_{2}(p))$.
Let $\chi >(L(p)/\varepsilon )^{1/p}$, $p>0$,
$\forall \varepsilon \in (0,1)$. Then applying the Chebyshev inequality
yields
\begin{equation*}
\limsup _{t\to \infty }\mathrm{P}\{|X(t)|>\chi \}\le
\frac{1}{\chi ^{p}}\lim \sup _{t\to \infty }E[|X(t)|^{p}]\le
\frac{L(p)}{\chi ^{p}}<\varepsilon .
\nonumber
\qedhere
\end{equation*}
\end{proof}

%s4 #&#
\section{Stochastic permanence\index{stochastic ! permanence}}
\label{sec4}

The property of stochastic permanence\index{stochastic ! permanence} is important since it means the long-time
survival in a population dynamics.

%d2 #&#
\begin{defin}%
(\cite{Li}) The solution $X(t)$ to the system \textup{(\ref{eq2})} is said to
be stochastically permanent\index{stochastically permanent} if for any $\varepsilon >0$ there are positive
constants $H=H(\varepsilon )$, $h=h(\varepsilon )$ such that
\begin{align}
\liminf \limits _{t\to \infty }\mathrm{P}\{x_{i}(t)\le H\}\ge 1-
\varepsilon ,\qquad \liminf \limits _{t\to \infty }\mathrm{P}\{x_{i}(t)
\ge h\}\ge 1-\varepsilon ,
\nonumber
\end{align}
for $i=1,2$ and for any inial value $X_{0}\in \mathbb{R}^{2}_{+}$.
\end{defin}

%t3 #&#
\begin{thm}
Let Assumption \ref{ass1} be fulfilled. If
\begin{align}
\min _{i=1,2}\inf _{t\ge 0}(a_{i\min }(t)-\beta _{i}(t))>0,\quad  \text{where}
\ a_{i\min }(t)=\min _{j=1,2}a_{ij}(t),\quad  i=1,2,
\nonumber
\end{align}
then the solution $X(t)$ to the system \textup{(\ref{eq2})} with the initial condition
$X_{0}\in \mathbb{R}^{2}_{+}$ is stochastically permanent.\index{stochastically permanent}
\end{thm}

\begin{proof}

For the process $U_{i}(t)=1/x_{i}(t)$, $i=1,2$, by the It\^{o} formula\index{It\^{o} formula} we
have
\begin{align}
U_{i}(t)=U_{i}(0)+\int \limits _{0}^{t} U_{i}(s) \left [
\rule{0pt}{20pt}-\frac{a_{i1}(s)+a_{i2}(s)x_{3-i}(s)}{1+x_{3-i}(s)}+c_{i}(s)x_{i}(s)+
\sigma _{i}^{2}(s)\right .
\nonumber
\\
\left .+\int \limits _{\mathbb{R}}
\frac{\gamma _{i}^{2}(s,z)}{1+\gamma _{i}(s,z)}\Pi _{1}(dz)\right ]ds -
\int \limits _{0}^{t}U_{i}(s)\sigma _{i}(s)dw_{i}(s)
\nonumber
\\
-\int \limits _{0}^{t}\!\!\!\!\int \limits _{\mathbb{R}}U_{i}(s)
\frac{\gamma _{i}(s,z)}{1+\gamma _{i}(s,z)}\tilde{\nu }_{1}(ds,dz)-
\int \limits _{0}^{t}\!\!\!\int \limits _{\mathbb{R}}U_{i}(s)
\frac{\delta _{i}(s,z)}{1+\delta _{i}(s,z)}\nu _{2}(ds,dz).
\nonumber
\end{align}

Then by the It\^{o} formula\index{It\^{o} formula} we derive for $0<\theta <1$:
\begin{align}
(1+U_{i}(t))^{\theta }=(1+U_{i}(0))^{\theta }+\int \limits _{0}^{t}
\theta (1+U_{i}(s))^{\theta -2}\left \{  \rule{0pt}{20pt}(1+U_{i}(s))U_{i}(s)
\right .
\nonumber
\\
\times \left [-\frac{a_{i1}(s)+a_{i2}(s)x_{3-i}(s)}{1+x_{3-i}(s)}+c_{i}(s)x_{i}(s)+
\sigma _{i}^{2}(s)
\phantom{\int \limits _{\mathbb{R}}}
\right .
\nonumber
\\
\left . +\int \limits _{\mathbb{R}}
\frac{\gamma _{i}^{2}(s,z)}{1+\gamma _{i}(s,z)}\Pi _{1}(dz)\right ] +
\frac{\theta -1}{2}U_{i}^{2}(s)\sigma _{i}^{2}(s)
\nonumber
\\
+ \frac{1}{\theta }\int \limits _{\mathbb{R}}\left [(1+U_{i}(s))^{2}
\left (\left (
\frac{1+U_{i}(s)+\gamma _{i}(s,z)}{(1+\gamma _{i}(s,z))(1+U_{i}(s))}
\right )^{\theta }-1\right )\right .
\nonumber
\\
+\left . \theta (1+U_{i}(s))
\frac{U_{i}(s)\gamma _{i}(s,z)}{1+\gamma _{i}(s,z)}\right ]\Pi _{1}(dz)
\nonumber
\end{align}
%
%e13 #&#
\begin{align}
\label{eq12}
\left . +\frac{1}{\theta }\int \limits _{\mathbb{R}}(1+U_{i}(s))^{2}
\left [\left (
\frac{1+U_{i}(s)+\delta _{i}(s,z)}{(1+\delta _{i}(s,z))(1+U_{i}(s))}
\right )^{\theta }-1\right ]\Pi _{2}(dz)\right \}  ds
\nonumber
\\
-\int \limits _{0}^{t}\theta (1+U_{i}(s))^{\theta -1}U_{i}(s)\sigma _{i}(s)dw_{i}(s)
\nonumber
\\
+ \int \limits _{0}^{t}\!\!\int \limits _{\mathbb{R}}\left [\left (1+
\frac{U_{i}(s)}{1+\gamma _{i}(s,z)}\right )^{\theta }-(1+U_{i}(s))^{\theta }\right ]\tilde{\nu }_{1}(ds,dz)
\nonumber
\\
+ \int \limits _{0}^{t}\!\!\int \limits _{\mathbb{R}}\left [\left (1+
\frac{U_{i}(s)}{1+\delta _{i}(s,z)}\right )^{\theta }-(1+U_{i}(s))^{\theta }\right ]\tilde{\nu }_{2}(ds,dz)= (1+U_{i}(0))^{\theta }
\nonumber
\\
+\int \limits _{0}^{t} \theta (1+U_{i}(s))^{\theta -2}J(s)ds-I_{1,\mathrm{stoch}}(t)+I_{2,\mathrm{stoch}}(t)+I_{3,\mathrm{stoch}}(t),
\end{align}
where $I_{i,\mathrm{stoch}}(t), i=\overline{1,3}$, are corresponding stochastic integrals\index{stochastic ! integrals}
in \textup{(\ref{eq12})}. Under the Assumption \ref{ass1} there exist constants
$|K_{1}(\theta )|<\infty $, $|K_{2}(\theta )|<\infty $ such that for the
process $J(t)$ we have the estimate
\begin{align}
J(t)\le (1+U_{i}(t))U_{i}(t)\left [-a_{i\min }(t)+c_{\max }U_{i}^{-1}(t)
+\sigma _{i}^{2}(t)\!\!
\phantom{\int \limits _{\mathbb{R}}}
\right .
\nonumber
\\
\left .+\int \limits _{\mathbb{R}}
\frac{\gamma _{i}^{2}(s,z)}{1+\gamma _{i}(s,z)}\Pi _{1}(dz)\right ] +
\frac{\theta -1}{2}U_{i}^{2}(s)\sigma _{i}^{2}(s)
\nonumber
\\
+\frac{1}{\theta }\int \limits _{\mathbb{R}}\left [(1+U_{i}(s))^{2}
\left (\left (\frac{1}{1+\gamma _{i}(s,z)}+ \frac{1}{1+U_{i}(s)}
\right )^{\theta }-1\right )\right .
\nonumber
\\
+\left . \theta (1+U_{i}(s))
\frac{U_{i}(s)\gamma _{i}(s,z)}{1+\gamma _{i}(s,z)}\right ]\Pi _{1}(dz)
\nonumber
\\
+\frac{1}{\theta }\int \limits _{\mathbb{R}}(1+U_{i}(s))^{2}\left [
\left (\frac{1}{1+\delta _{i}(s,z)}+ \frac{1}{1+U_{i}(s)}\right )^{
\theta }-1\right ]\Pi _{2}(dz)
\nonumber
\\
\le U_{i}^{2}(t)\left [-a_{i\min }(t)+\frac{\sigma _{i}^{2}(t)}{2} +
\int \limits _{\mathbb{R}}\gamma _{i}(t,z)\Pi _{1}(dz)+
\frac{\theta }{2}\sigma _{i}^{2}(t)\right .
\nonumber
\\
\left .+\frac{1}{\theta }\int \limits _{\mathbb{R}}[(1+\gamma _{i}(t,z))^{-
\theta }-1]\Pi _{1}(dz) +\frac{1}{\theta }\int \limits _{\mathbb{R}}[(1+
\delta _{i}(t,z))^{-\theta }-1]\Pi _{2}(dz)\right ]
\nonumber
\\
+K_{1}(\theta )U_{i}(t)+K_{2}(\theta ).
\nonumber
\end{align}
Here we used the inequality
$(x+y)^{\theta }\le x^{\theta }+\theta x^{\theta -1}y$,
$0<\theta <1$, $x,y>0$. Due to
\begin{align}
\lim _{\theta \to 0+}\left [\frac{\theta }{2}\sigma _{i}^{2}(t)+
\frac{1}{\theta }\int \limits _{\mathbb{R}}[(1+\gamma _{i}(t,z))^{-
\theta }-1]\Pi _{1}(dz)\right .
\nonumber
\\
\left . +\frac{1}{\theta }\int \limits _{\mathbb{R}}[(1+\delta _{i}(t,z))^{-
\theta }-1]\Pi _{2}(dz)\right ]
\nonumber
\\
=-\int \limits _{\mathbb{R}}\ln (1+\gamma _{i}(t,z))\Pi _{1}(dz)-
\int \limits _{\mathbb{R}}\ln (1+\delta _{i}(t,z))\Pi _{2}(dz),
\nonumber
\end{align}
and the condition
$\min _{i=1,2}\inf _{t\ge 0}(a_{i\min }(t)-\beta _{i}(t))>0$ we can choose
a sufficiently small $0<\theta <1$ to satisfy
\begin{align}
K_{0}(\theta )=\min _{i=1,2}\inf _{t\ge 0}\left \{  a_{i\min }(t)-
\frac{\sigma _{i}^{2}(t)}{2}-\int \limits _{\mathbb{R}}\gamma _{i}(t,z)
\Pi _{1}(dz)-\frac{\theta }{2}\sigma _{i}^{2}(t)\right .
\nonumber
\\
\left . - \frac{1}{\theta }\int \limits _{\mathbb{R}}[(1+\gamma _{i}(t,z))^{-
\theta }-1]\Pi _{1}(dz) -\frac{1}{\theta }\int \limits _{\mathbb{R}}[(1+
\delta _{i}(t,z))^{-\theta }-1]\Pi _{2}(dz)\right \}  >0.
\nonumber
\end{align}

So from \textup{(\ref{eq12})} and the estimate for $J(t)$ we derive
%
%e14 #&#
\begin{align}
\label{eq13}
d\left [(1+U_{i}(t))^{\theta }\right ]\le \theta (1+U_{i}(t))^{\theta -2}[-K_{0}(
\theta )U_{i}^{2}(t)+K_{1}(\theta )U_{i}(t)+K_{2}(\theta )]dt
\nonumber
\\
-\theta (1+U_{i}(t))^{\theta -1}U_{i}(t)\sigma _{i}(t)dw_{i}(t)+
\int \limits _{\mathbb{R}}\left [\left (1+
\frac{U_{i}(t)}{1+\gamma _{i}(t,z)}\right )^{\theta }\right .
\nonumber
\\
\left . -(1 +U_{i}(t))^{\theta }\rule{0pt}{18pt}\right ]\tilde{\nu }_{1}(dt,dz)
+ \int \limits _{\mathbb{R}}\left [\left (1+
\frac{U_{i}(t)}{1+\delta _{i}(t,z)}\right )^{\theta } -(1+U_{i}(t))^{\theta }\right ]\tilde{\nu }_{2}(dt,dz).
\end{align}
By the It\^{o} formula\index{It\^{o} formula} and \textup{(\ref{eq13})} we have
%
%e15 #&#
\begin{align}
\label{eq14}
d\left [e^{\lambda t}(1+U_{i}(t))^{\theta }\right ]=\lambda e^{\lambda t}(1+U_{i}(t))^{\theta }dt+e^{\lambda t} d\left [(1+U_{i}(t))^{\theta }\right ]
\nonumber
\\
\le e^{\lambda t}\theta (1+U_{i}(t))^{\theta -2}\left [-\left (K_{0}(
\theta )-\frac{\lambda }{\theta }\right )U_{i}^{2}(t)+\left (K_{1}(
\theta )+\frac{2\lambda }{\theta }\right )U_{i}(t)\right .
\nonumber
\\
\left . +K_{2}(\theta )+\frac{\lambda }{\theta }\right ]dt-\theta e^{
\lambda t}(1+U_{i}(t))^{\theta -1}U_{i}(t)\sigma _{i}(t)dw_{i}(t)
\nonumber
\\
+e^{\lambda t}\int \limits _{\mathbb{R}}\left [\left (1+
\frac{U_{i}(t)}{1+\gamma _{i}(t,z)}\right )^{\theta }-(1+U_{i}(t))^{\theta }\right ]\tilde{\nu }_{1}(dt,dz)
\nonumber
\\
\displaystyle + e^{\lambda t}\int \limits _{\mathbb{R}}\left [\left (1+
\frac{U_{i}(t)}{1+\delta _{i}(t,z)}\right )^{\theta }-(1+U_{i}(t))^{\theta }\right ]\tilde{\nu }_{2}(dt,dz).
\end{align}

Let us choose $\lambda >0$ such that
$K_{0}(\theta )-\lambda /\theta >0$. Then the function
\begin{align}
(1+U_{i}(t))^{\theta -2}\left [-\left (K_{0}(\theta )-
\frac{\lambda }{\theta }\right )U_{i}^{2}(t)+\left (K_{1}(\theta )+
\frac{2\lambda }{\theta }\right )U_{i}(t) +K_{2}(\theta )+
\frac{\lambda }{\theta }\right ]
\nonumber
\end{align}
is bounded from above by some constant $K>0$. So by integrating
\textup{(\ref{eq14})} and taking the expectation we obtain
%
%e16 #&#
\begin{align}
\label{eq15}
e^{\lambda t}\mathrm{E}\left [(1+U_{i}(t))^{\theta }\right ]\le \left (1+
\frac{1}{x_{i0}}\right )^{\theta }+\frac{\lambda }{\theta }K\left (e^{
\lambda t}-1\right ),
\end{align}
because the expectation of stochastic integrals\index{stochastic ! integrals} are equal zero by
\textup{(\ref{eq10})} and
\begin{align}
\mathrm{E}\left [\int \limits _{0}^{t} e^{2\lambda s}(1+U_{i}(s))^{2
\theta -2}U^{2}_{i}(s)\sigma _{i}^{2}(s)ds\right ]=\mathrm{E}\left [
\int \limits _{0}^{t}
\frac{e^{2\lambda s}\sigma _{i}^{2}(s)}{(1+x^{-1}_{i}(s))^{2-2\theta }x^{2}_{i}(s)}ds
\right ]
\nonumber
\\
\le \int \limits _{0}^{t} e^{2\lambda s}\sigma _{i}^{2}(s)\mathrm{E}[x_{i}^{2
\theta }]ds<\infty ,
\nonumber
\\
\mathrm{E}\left [\int \limits _{0}^{t}\!\!\int \limits _{\mathbb{R}}
e^{2\lambda s}\left [\left (1+\frac{U_{i}(s)}{1+\gamma _{i}(s,z)}
\right )^{\theta }-(1+U_{i}(s))^{\theta }\right ]^{2}\Pi _{1}(dz)ds
\right ]
\nonumber
\\
\le L_{1}\int \limits _{0}^{t}e^{2\lambda s}\mathrm{E}\left [x_{i}^{2
\theta }(s)\right ]ds<\infty ,
\nonumber
\\
\mathrm{E}\left [\int \limits _{0}^{t}\!\!\int \limits _{\mathbb{R}}
e^{2\lambda s}\left [\left (1+\frac{U_{i}(s)}{1+\delta _{i}(s,z)}
\right )^{\theta }-(1+U_{i}(s))^{\theta }\right ]^{2}\Pi _{2}(dz)ds
\right ]
\nonumber
\\
\le L_{2}\int \limits _{0}^{t}e^{2\lambda s}\mathrm{E}\left [x_{i}^{2
\theta }(s)\right ]ds<\infty ,
\nonumber
\end{align}
where
\begin{align}
L_{1}=\theta ^{2} \max \left \{
\frac{(\gamma ^{2})_{\max }}{(1+\gamma _{\min })^{2}},
\frac{(\gamma ^{2})_{\max }}{(1+\gamma _{\min })^{2\theta }}\right \}
\Pi _{1}(\mathbb{R})<\infty ,
\nonumber
\\
L_{2}=\theta ^{2} \max \left \{
\frac{(\delta ^{2})_{\max }}{(1+\delta _{\min })^{2}},
\frac{(\delta ^{2})_{\max }}{(1+\delta _{\min })^{2\theta }}\right \}
\Pi _{2}(\mathbb{R})<\infty .
\nonumber
\end{align}

From \textup{(\ref{eq15})} we obtain
%
%e17 #&#
\begin{align}
\label{eq16}
\limsup _{t\to \infty }\mathrm{E}\left [\left (\frac{1}{x_{i}(t)}
\right )^{\theta }\right ]=\limsup _{t\to \infty }\mathrm{E}\left [U_{i}^{
\theta }(t)\right ]
\nonumber
\\
\le \limsup _{t\to \infty }\mathrm{E}\left [(1+U_{i}(t))^{\theta }
\right ]\le \frac{\theta K}{\lambda },\quad  i=1,2.
\end{align}

From \textup{(\ref{eq11})} and \textup{(\ref{eq16})} by the Chebyshev inequality we
can derive that for an arbitrary $\varepsilon \in (0,1)$ there are
%a pair of
positive constants $H=H(\varepsilon )$ and $h=h(\varepsilon )$ such
that
\begin{equation*}
\liminf \limits _{t\to \infty }\mathrm{P}\{x_{i}(t)\le H\}\ge 1-
\varepsilon ,\qquad \liminf \limits _{t\to \infty }\mathrm{P}\{x_{i}(t)
\ge h\}\ge 1-\varepsilon ,\quad  i=1,2.
\nonumber
\qedhere \end{equation*}
\end{proof}
%

%s5 #&#
\section{Extinction,\index{extinction} nonpersistence and strong persistence in the mean}
\label{sec5}

The property of extinction\index{extinction} in the stochastic models of population dynamics
means that every species will become extinct with probability $1$.

%d3 #&#
\begin{defin}
The solution $X(t)=(x_{1}(t),x_{2}(t))$, $t\ge 0$, to the system
\textup{(\ref{eq2})} is said to be extinct\index{extinct} if for every initial data
$X_{0}>0$ we have $\lim _{t\to \infty }x_{i}(t)=0$ almost surely (a.s.),
$i=1,2$.
\end{defin}

%t4 #&#
\begin{thm}
Let Assumption \ref{ass1} be fulfilled. If
\begin{align}
{\bar{p}}^{*}_{i}=\limsup _{t\to \infty }\frac{1}{t}\int \limits _{0}^{t}p_{i
\max }(s)ds<0,\quad  \text{where}\ p_{i\max }(s)=a_{i\max }(s)-\beta _{i}(s),
\nonumber
\end{align}
$a_{i\max }(t)=\max _{j=1,2}a_{ij}(t)$, $i=1,2$, then the solution
$X(t)$ to the system \textup{(\ref{eq2})} with the initial condition
$X_{0}\in \mathbb{R}^{2}_{+}$ will be extinct.\index{extinct}
\end{thm}

\begin{proof}
By the It\^{o} formula,\index{It\^{o} formula} we have
%
%e18 #&#
\begin{align}
\label{eq17}
\ln x_{i}(t)=\ln x_{i0}+\int \limits _{0}^{t}\left [\!\!\!
\phantom{\int \limits _{\mathbb{R}}}
\frac{a_{i1}(s)+a_{i2}(s)x_{3-i}(s)}{1+x_{3-i}(s)}-c_{i}(s)x_{i}(s)-
\frac{\sigma _{i}^{2}(s)}{2}\right .
\nonumber
\\
\left .+ \int \limits _{\mathbb{R}}[\ln (1+\gamma _{i}(s,z))-\gamma _{i}(s,z)]
\Pi _{1}(dz)+\int \limits _{\mathbb{R}}\ln (1+\delta _{i}(s,z))\Pi _{2}(dz)
\right ]ds
\nonumber
\\
+ \int \limits _{0}^{t}\sigma _{i}(s)dw_{i}(s)+\int \limits _{0}^{t}
\!\!\int \limits _{\mathbb{R}}\ln (1+\gamma _{i}(s,z))\tilde{\nu }_{1}(ds,dz)
\nonumber
\\
+ \int \limits _{0}^{t}\!\!\int \limits _{\mathbb{R}}\ln (1+\delta _{i}(s,z))
\tilde{\nu }_{2}(ds,dz)\le \ln x_{i0}+\int \limits _{0}^{t}p_{i\max }(s)ds+M(t),
\end{align}
where the martingale
%
%e19 #&#
\begin{align}
\label{eq18}
M(t)=\int \limits _{0}^{t}\sigma _{i}(s)dw_{i}(s)+\int \limits _{0}^{t}
\!\!\int \limits _{\mathbb{R}}\ln (1+\gamma _{i}(s,z))\tilde{\nu }_{1}(ds,dz)
\nonumber
\\
+ \int \limits _{0}^{t}\!\!\int \limits _{\mathbb{R}}\ln (1+\delta _{i}(s,z))
\tilde{\nu }_{2}(ds,dz)
\end{align}
has quadratic variation
\begin{align}
\langle M,M\rangle (t)=\int \limits _{0}^{t}\sigma ^{2}_{i}(s)ds+
\int \limits _{0}^{t}\!\!\int \limits _{\mathbb{R}}\ln ^{2}(1+
\gamma _{i}(s,z))\Pi _{1}(dz)ds
\nonumber
\\
+ \int \limits _{0}^{t}\!\!\int \limits _{\mathbb{R}}\ln ^{2}(1+
\delta _{i}(s,z))\Pi _{2}(dz)ds\le Kt.
\nonumber
\end{align}

Then the strong law of large numbers for local martingales (\cite{Lip})
yields\break $\lim _{t\to \infty }M(t)/t=0$ a.s. Therefore, from
\textup{(\ref{eq17})} we have
\begin{equation*}
\limsup _{t\to \infty }\frac{\ln x_{i}(t)}{t}\le \limsup _{t\to
\infty }\frac{1}{t}\int \limits _{0}^{t} p_{i\max }(s)ds<0, \quad
\text{a.s.}
\end{equation*}
So $\lim _{t\to \infty }x_{i}(t)=0$, $i=1,2$, a.s.
\end{proof}

%d4 #&#
\begin{defin}%
(\cite{Liu}) The solution $X(t)=(x_{1}(t),x_{2}(t))$, $t\ge 0$, to the system
\textup{(\ref{eq2})} is said to be nonpersistence in the mean if for every initial
data $X_{0}>0$ we have
$\lim _{t\to \infty }\frac{1}{t}\int \limits _{0}^{t}x_{i}(s)ds=0$ a.s.,
$i=1,2$.%
\end{defin}

%t5 #&#
\begin{thm}
Let Assumption \ref{ass1} be fulfilled. If
${\bar{p}}_{i}^{*}=0$, $i=1,2$, then the solution $X(t)$ to the system
\textup{(\ref{eq2})} with the initial condition $X_{0}\in \mathbb{R}^{2}_{+}$ will
be nonpersistence in the mean.
\end{thm}

\begin{proof}
From the first equality in \textup{(\ref{eq17})} we have
%
%e20 #&#
\begin{align}
\label{eq19}
\ln x_{i}(t)\le \ln x_{i0}+\int \limits _{0}^{t} p_{i\max }(s)ds -c_{
\min }\int \limits _{0}^{t}x_{i}(s)ds+M(t),
\end{align}
where the martingale $M(t)$\index{martingale} is defined in \textup{(\ref{eq18})}. From the definition
of ${\bar{p}}^{*}_{i}$ and the strong law of large numbers for $M(t)$ it
follows that $\forall \varepsilon >0$, $\exists t_{0}\ge 0$ such that
\begin{align}
\frac{1}{t}\int \limits _{0}^{t} p_{i\max }(s)ds\le {\bar{p}}^{*}_{i}+
\frac{\varepsilon }{2},\ \frac{M(t)}{t}\le \frac{\varepsilon }{2},\quad  \forall t\ge t_{0},\ \mathrm{a.s.}
\nonumber
\end{align}
So, from \textup{(\ref{eq19})} we derive
%
%e21 #&#
\begin{align}
\label{eq20}
\ln x_{i}(t)- \ln x_{i0}\le t({\bar{p}}^{*}_{i}+\varepsilon )-c_{\min }
\int \limits _{0}^{t}x_{i}(s)ds=t\varepsilon -c_{\min }\int \limits _{0}^{t}x_{i}(s)ds.
\end{align}
Let $y_{i}(t) =\int _{0}^{t}x_{i}(s)ds$, then from \textup{(\ref{eq20})} we have
\begin{align}
\ln \left (\frac{dy_{i}(t)}{dt}\right )\le \varepsilon t-c_{\min } y_{i}(t)+
\ln x_{i0} \Rightarrow e^{c_{\min }y_{i}(t)}\frac{dy_{i}(t)}{dt}\le x_{i0}e^{
\varepsilon t}.
\nonumber
\end{align}
Integrating the last inequality from $t_{0}$ to $t$ yields
\begin{align}
e^{c_{\min }y_{i}(t)}\le \frac{c_{\min }x_{i0}}{\varepsilon }\left (e^{
\varepsilon t}-e^{\varepsilon t_{0}}\right ) +e^{c_{\min }y_{i}(t_{0})},
\quad  \forall t\ge t_{0},\ \mathrm{a.s.}
\nonumber
\end{align}
So
\begin{align}
y_{i}(t)\le \frac{1}{c_{\min }}\ln \left [e^{c_{\min }y_{i}(t_{0})}+
\frac{c_{\min }x_{i0}}{\varepsilon }\left (e^{\varepsilon t}-e^{
\varepsilon t_{0}}\right ) \right ], \quad  \forall t\ge t_{0},\ \mathrm{a.s.},
\nonumber
\end{align}
and therefore
\begin{align}
\limsup _{t\to \infty }\frac{1}{t}\int \limits _{0}^{t}x_{i}(s)ds\le
\frac{\varepsilon }{c_{\min }}, \ \ \mathrm{a.s.}
\nonumber
\end{align}
Since $\varepsilon >0$ is arbitrary and
$X(t)\in \mathbb{R}^{2}_{+}$, we have
\begin{equation*}
\lim _{t\to \infty }\frac{1}{t}\int \limits _{0}^{s}x_{i}(s)ds=0,\quad  i=1,2,
\ \mathrm{a.s.}
\qedhere
\end{equation*}
\end{proof}

%d5 #&#
\begin{defin}
(\cite{Liu}) The solution $X(t)=(x_{1}(t),x_{2}(t))$, $t\ge 0$, to the system
\textup{(\ref{eq2})} is said to be strongly persistence\index{persistence} in the mean if for every
initial data $X_{0}>0$ we have
$\liminf _{t\to \infty }\frac{1}{t}\int \limits _{0}^{t}x_{i}(s)ds>0$ a.s.,
$i=1,2$.%
\end{defin}

%t6 #&#
\begin{thm}
Let Assumption \ref{ass1} be fulfilled. If
${\bar{p}}_{i*}=\liminf _{t\to \infty }\frac{1}{t}\int \limits _{0}^{t}p_{i
\min }(s)ds>0$, {where} $p_{i\min }(s)=a_{i\min }(s)-\beta _{i}(s)$,
$a_{i\min }(t)=\min _{j=1,2}a_{ij}(t)$, $i=1,2$, then
\begin{align}
\liminf _{t\to \infty }\frac{1}{t}\int \limits _{0}^{t}x_{i}(s)ds\ge
\frac{{\bar{p}}_{i*}}{c_{i\sup }}.
\nonumber
\end{align}
Therefore, the solution $X(t)$ to the system \textup{(\ref{eq2})} with the initial
condition $X_{0}\in \mathbb{R}^{2}_{+}$ will be strongly persistence\index{persistence} in
the mean.
\end{thm}

\begin{proof}
From the first equality in \textup{(\ref{eq17})} we have
%
%e22 #&#
\begin{align}
\label{eq21}
\ln x_{i}(t)\ge \ln x_{i0}+\int \limits _{0}^{t} p_{i\min }(s)ds -c_{i
\sup }\int \limits _{0}^{t}x_{i}(s)ds+M(t),
\end{align}
where the martingale $M(t)$\index{martingale} is defined in \textup{(\ref{eq18})}. From the definition
of ${\bar{p}}_{i*}$ and the strong~law of large numbers for $M(t)$ it follows
that $\forall \varepsilon >0$, $\exists t_{0}\ge 0$ such that
$\frac{1}{t}\int \limits _{0}^{t} p_{i\min }(s)ds\ge {\bar{p}}_{i*}-
\frac{\varepsilon }{2}$, $\frac{M(t)}{t}\ge -\frac{\varepsilon }{2}$, $\forall t\ge t_{0}$, a.s. So, from \textup{(\ref{eq21})} we obtain
%
%e23 #&#
\begin{align}
\label{eq22}
\ln x_{i}(t)\ge \ln x_{i0}+ t({\bar{p}}_{i*}-\varepsilon )-c_{i\sup }
\int \limits _{0}^{t}x_{i}(s)ds.
\end{align}
Let us choose sufficiently small $\varepsilon >0$ such that
${\bar{p}}_{i*}-\varepsilon >0$.

Let $y_{i}(t) =\int _{0}^{t}x_{i}(s)ds$, then from \textup{(\ref{eq22})} we have
\begin{align}
\ln \left (\frac{dy_{i}(t)}{dt}\right )\ge ({\bar{p}}_{i*}-
\varepsilon ) t-c_{i\sup } y_{i}(t)+\ln x_{i0}.
\nonumber
\end{align}
Hence
$e^{c_{i\sup }y_{i}(t)}\frac{dy_{i}(t)}{dt}\ge x_{i0}e^{( {\bar{p}}_{i*}-
\varepsilon ) t}$. Integrating the last inequality from $t_{0}$ to $t$ yields
\begin{align}
e^{c_{i\sup }y_{i}(t)}\ge
\frac{c_{i\sup }x_{i0}}{{\bar{p}}_{i*}-\varepsilon }\left (e^{({\bar{p}}_{i*}-
\varepsilon ) t}-e^{({\bar{p}}_{i*}-\varepsilon ) t_{0}}\right ) +e^{c_{i
\sup }y_{i}(t_{0})},\quad  \forall t\ge t_{0}, \ \mathrm{a.s.}
\nonumber
\end{align}
So
\begin{align}
y_{i}(t)\ge \frac{1}{c_{i\sup }}\ln \left [e^{c_{i\sup }y_{i}(t_{0})}+
\frac{c_{i\sup }x_{i0}}{{\bar{p}}_{i*}-\varepsilon }\left (e^{({\bar{p}}_{i*}-
\varepsilon ) t}-e^{({\bar{p}}_{i*}-\varepsilon ) t_{0}}\right )
\right ],\ \mathrm{a.s.},
\nonumber
\end{align}
$\forall t\ge t_{0}$, and therefore
\begin{align}
\liminf _{t\to \infty }\frac{1}{t}\int \limits _{0}^{t}x_{i}(s)ds\ge
\frac{({\bar{p}}_{i*}-\varepsilon )}{c_{i\sup }}, \ \ \mathrm{a.s.}
\nonumber
\end{align}
Using the arbitrariness of $\varepsilon >0$ we get the assertion of the
theorem.
\end{proof}

%\begin{appendix}
%\end{appendix}

%\begin{acknowledgement}%[title={Acknowledgments}]
%\end{acknowledgement}

%\begin{funding}
%\gsponsor[id=,sponsor-id=]{}
%\gnumber[refid=]{}
%\end{funding}

% structpyb loaded by romualda, 2020-02-25 12:50:34

%

\end{document}